\RequirePackage{fix-cm}
\documentclass[smallcondensed]{svjour3}     
\smartqed  
\usepackage{mathptmx}      
\usepackage{amsmath}
\usepackage{latexsym}
\usepackage{graphicx}
\usepackage{xhfill}
\newtheorem{algorithm}{Algorithm}

\begin{document}
\newcommand{\xfill}[2][1ex]{{%
  \dimen0=#2\advance\dimen0 by #1
  \leaders\hrule height \dimen0 depth -#1\hfill%
}}
\title{Parallel extragradient-proximal methods for split equilibrium problems
}

\titlerunning{Parallel extragradient-proximal methods for split equilibrium problems}        

\author{Dang Van Hieu
}

\authorrunning{D. V. Hieu} 

\institute{Dang Van Hieu, Department of Mathematics, Vietnam National University \at
               334 - Nguyen Trai Street, Ha Noi, Viet Nam\\
              Tel.: +84-979817776\\
              \email{dv.hieu83@gmail.com}           
}

\date{Received: date / Accepted: date}

\maketitle

\begin{abstract}
In this paper, we introduce two parallel extragradient-proximal methods for solving split equilibrium problems. The algorithms combine 
the extragradient method, the proximal method and the hybrid (outer approximation) method. The weak and strong convergence theorems for iterative 
sequences generated by the algorithms are established under widely used assumptions for equilibrium bifunctions.
\end{abstract}
\section{Introduction}\label{intro}
Let $H_1,H_2$ be two real Hilbert spaces and $C,Q$ be two nonempty closed convex subsets of $H_1,H_2$, respectively. 
Let $A:H_1\to H_2$ be a bounded linear operator. Let $f:C\times C\to \Re$ and $F:Q\times Q\to \Re$ be two bifunctions with 
$f(x,x)=0$ for all $x\in C$ and $F(y,y)=0$ for all $y\in Q$. The split equilibrium problem (SEP) \cite{H2012} is stated as follows:
\begin{equation}\label{SEP}
\begin{cases}
\mbox{Find}~x^*\in C~\mbox{such that} ~f(x^*,y)\ge 0,~\forall y\in C,\\
\mbox{and}~u^*=Ax^*\in Q~\mbox{solves} ~F(u^*,u)\ge 0~\forall u\in Q.
\end{cases}
\end{equation}
Obviously, if $F=0$ then SEP becomes the following equilibrium problem (EP) \cite{BO1994}.
\begin{equation}\label{EP}
\mbox{Find}~x^*\in C~\mbox{such that} ~f(x^*,y)\ge 0,~\forall y\in C.
\end{equation}
The solution set of EP (\ref{EP}) for the bifunction $f$ on $C$ is denoted by $EP(f,C)$. SEP is very general in the sense that it includes many 
mathematical models as: split optimization problems, split fixed point problems, split 
variational inclusion problems, split variational inequality problems \cite{CE1994,CGR2012,CS2009,KR2013,KR2014,KS2014,M2011,M2010}. 
Split problems describe finding a solution of a problem whose image under a bounded linear transformation is a solution of another problem. 
A special case of SEP in practice is the split convex feasibility problem which had been studied and used as a
model in intensity-modulated radiation therapy treatment planning, see \cite{CBMT2006,CEKB2005}. 

Some algorithms for solving SEP can be found, for instance, in \cite{DKK2014,H2012,HD2014,KR2013,KR2014}. Almost proposed methods for 
SEPs based on the proximal method \cite{K2000} which consists of solving a regularized equilibrium problem, i.e., at current iteration, given 
$x_n$, the next iterate $x_{n+1}$ solves the following problem
 $$\mbox{Find}~x\in C~\mbox{such that}~ f(x,y)+\frac{1}{r_n}\left\langle y-x,x-x_n\right\rangle\ge 0,~\forall y\in C $$
or $x_{n+1}=T_{r_n}^f(x_n)$ where $T_{r_n}^f$ is the resolvent of the bifunction $f$ and $r_n>0$, see \cite{CH2005}. 
 In 2012, He \cite{H2012} used the proximal method and proposed the following algorithm 
$$
\begin{cases}
f_i(u_n^i,y)+\frac{1}{r_n}\left\langle y-u_n^i,u_n^i-x_n\right\rangle\ge 0,~\forall y\in C, i=1,\ldots,N,\\
\tau_n=\frac{u_n^1+\ldots+u_n^N}{N},\\
F(w_n,z)+\frac{1}{r_n}\left\langle z-w_n,w_n-\tau_n\right\rangle\ge 0,~\forall z\in Q,\\
x_{n+1}=P_C(\tau_n+\mu A^*(w_n-A\tau_n))
\end{cases}
$$
for finding an element in $\Omega=\left\{p\in \cap_{i=1}^k EP(f_i,C):Ap\in EP(F,Q)\right\}$. Under the assumption of 
the monotonicity of $f_i:C\times C\to \Re, ~F:Q\times Q\to \Re$ and suitable conditions on the parameters $r_n,\mu$, 
the author proved that $\left\{u_n^i\right\}$, $\left\{x_n\right\}$ converge weakly to some point in $\Omega$.

Very recently, for finding a common solution of a system of equilibrium problems for pseudomonotone monotone and Lipschitz-type 
continuous bifunctions $\left\{f_i\right\}_{i=1}^N$, the authors in \cite{HMA2015} have proposed the following parallel hybrid 
extragradient algorithm (also, see \cite{H2015a})
$$
\left\{
\begin{array}{ll}
&y_n^i = {\rm argmin} \{ \lambda f_i(x_n, y) +\frac{1}{2}||x_n-y||^2:  y \in C\} \quad i=1,\ldots,N,\\
&z_n^i = {\rm argmin} \{ \lambda f_i(y_n^i, y) +\frac{1}{2}||x_n-y||^2:  y \in C\} \quad i=1,\ldots,N,\\
&i_n = {\rm argmax}\{||z_n^i - x_n||: i =1,\ldots,N\},\bar{z}_n:=z^{i_n}_n,\\
&C_n = \{v \in C: ||\bar{z}_n - v|| \leq ||x_n-v||\},\\
&Q_n = \{ v \in C:  \langle x_0 - x_n, v - x_n \rangle \leq 0 \},\\
&x_{n+1}=P_{C_n\bigcap Q_n}x_0,n\ge 0.
\end{array}
\right.
$$
It has been proved that $\left\{x_n\right\}$, $\left\{y_n^i\right\}$, $\left\{z^i_n\right\}$ converge strongly to 
the projection of the starting point $x_0$ onto the solution set $\cap_{i=1}^N EP(f_i,C)$ under certain 
conditions on the parameter $\lambda$. The advantages of the extragradient method are that it is used for the 
class of pseudomonotone bifunctions and two optimization programs are solved at each iteration which seems 
to be numerically easier than the non-linear inequality in the proximal method, see for instance 
\cite{NSN2013,QMH2008,SNN2013} and the references therein.

In this paper, motivated by the recent works \cite{CGR2012,DKK2014,KR2013,KR2014} and the results above, we propose two 
parallel extragradient-proximal methods for SEPs for a finite family of bifunctions 
$\left\{f_i\right\}_{i=1}^N:C\times C\to \Re$ in $H_1$ and a system of bifunctions $\left\{F_j\right\}_{j=1}^M:Q\times Q\to\Re$ 
in $H_2$. In the first algorithm, we use the extragradient method for pseudomonotone EPs in $H_1$ and the proximal method for monotone EPs 
in $H_2$ to design the weak convergence algorithm. In order to obtain the strong convergence, we combine the first one with the hybrid method 
in the second algorithm. Under widely used assumptions for bifunctions, the convergence theorems are proved.

The paper is organized as follows: In Section \ref{pre}, we collect some definitions and preliminary results for the further use. Section \ref{main} 
deals with proposing and analyzing the convergence of the algorithms. 
\section{Preliminaries}\label{pre}
Let $C$ be a nonempty closed convex subset of a real Hilbert space $H$ with the inner product $\left\langle .,.\right\rangle$ and the induced norm 
$||.||$. We begin with some concepts of the monotonicity of a bifunction.
\begin{definition}\cite{BO1994,MO1992} A bifunction $f:C\times C\to \Re$ is said to be
\begin{itemize}
\item [$\rm i.$] strongly monotone on $C$ if there exists a constant $\gamma>0$ such that
$$ f(x,y)+f(y,x)\le -\gamma ||x-y||^2,~\forall x,y\in C; $$
\item [$\rm ii.$] monotone on $C$ if 
$$ f(x,y)+f(y,x)\le 0,~\forall x,y\in C; $$
\item [$\rm iii.$] pseudomonotone on $C$ if 
$$ f(x,y)\ge 0 \Longrightarrow f(y,x)\le 0,~\forall x,y\in C;$$
\item [$\rm iv.$] Lipschitz-type continuous on $C$ if there exist two positive constants $c_1,c_2$ such that
$$ f(x,y) + f(y,z) \geq f(x,z) - c_1||x-y||^2 - c_2||y-z||^2, ~ \forall x,y,z \in C.$$
\end{itemize}
\end{definition}
From the definitions above, it is clear that a strongly monotone bifunction is monotone and a monotone bifunction is pseudomonotone, i.e., 
$i.\Longrightarrow ii. \Longrightarrow iii.$ For solving SEP $(\ref{SEP})$, we assume that the bifunctions $f:C\times C\to \Re$ and 
$F:Q\times Q\to \Re$ satisfy the following Condition 1 and 
Condition 2, respectively.\\
\textbf{Condition 1}
\begin{itemize}
\item[\rm (A1)] $f$ is pseudomonotone on $C$ and $f(x,x)=0$ for all $x\in C$;
\item [\rm (A2)]  $f$ is Lipschitz-type continuous on $C$ with the constants $c_1,c_2$;
\item [\rm (A3)]   $f$ is jointly weakly continuous on $C\times C$ in the sense that, if $x,y\in C$ and $\left\{x_n\right\},\left\{y_n\right\}$ 
converge weakly to $x,y$, respectively, then $f(x_n,y_n)\to f(x,y)$ as $n\to\infty$;
\item [\rm (A4)]  $f(x,.)$ is convex and subdifferentiable on $C$  for every fixed $x\in C$
\end{itemize}
\textbf{Condition 2}
\begin{itemize}
\item[$\rm (\bar{A}1)$] $F$ is monotone on $C$ and $F(x,x)=0$ for all $x\in C$;
\item[$\rm (\bar{A}2)$] For all $x,y,z \in C$,
$$ \lim_{t\to 0^+}\sup F(tz+(1-t)x,y) \le F(x,y); $$
\item[$\rm (\bar{A}3)$] For all $x\in C$, $F(x,.)$ is convex and lower semicontinuous.
\end{itemize}

The following results concern with the monotone befunction $F$.
\begin{lemma}\label{ExitenceN0} \cite[Lemma 2.12]{CH2005} Let $C$ be a
closed and convex subset of a Hilbert space H, $F$ be a bifunction from $C\times C$ to
$\Re$ satisfying Condition 2 and let $r>0$,
$x\in H$. Then, there exists $z\in C$ such that
\begin{eqnarray*}
F(z,y)+\frac{1}{r}\langle y-z,z-x\rangle\geq0, \quad \forall y\in C.
\end{eqnarray*}
\end{lemma}
\begin{lemma}\label{lem.CH2005}\cite[Lemma 2.12]{CH2005} Let
$C$ be a closed and convex subset of a Hilbert space $H$, $F$ be a bifunction from
$C\times C$ to $\Re$ satisfying Condition 2.
For all $r>0$ and $x\in H$, define the mapping
\begin{eqnarray*}
T_r^F x=\{z\in C:F(z,y)+\frac{1}{r}\langle y-z,z-x\rangle\geq0, \quad
\forall y\in C\}.
\end{eqnarray*}
Then the following hold:

{\rm (B1)} $T_r^F$ is single-valued;

{\rm (B2)} $T_r^F$ is a firmly nonexpansive, i.e., for
all $x, y\in H,$
\begin{eqnarray*}
||T_r^Fx-T_r^Fy||^2\leq\langle T_r^Fx-T_r^Fy,x-y\rangle;
\end{eqnarray*}

{\rm (B3)} $Fix(T_r^F)=EP(F,C)$, where $Fix(T_r^F)$ is the fixed point set of $T_r^F$;

{\rm (B4)} $EP(F,C)$ is closed and convex.
\end{lemma}
\begin{lemma}\cite[Lemma 2.5]{H2012}\label{lem.H2012}
For $r,s>0$ and $x,y\in H$. Under the assumptions of Lemma \ref{lem.CH2005}, then 
$$ ||T_r^F(x)-T_s^F(y)||\le||x-y||+\frac{|s-r|}{s}||T_s^F(y)-y||.$$
\end{lemma}
The metric projection $P_C:H\to C$ is defined by $P_C x=\underset{y\in C}{\arg\min}\left\{\left\|y-x\right\|\right\}$. 
It is well-known that $P_C$ has the following characteristic properties, see \cite{GR1984} for more details.
\begin{lemma}\label{lem.PC}
Let $P_C:H\to C$ be the metric projection from $H$ onto $C$. Then
\begin{itemize}
\item [$\rm i.$] For all $x\in C, y\in H$,
\begin{equation}\label{eq:ProperOfPC}
\left\|x-P_C y\right\|^2+\left\|P_C y-y\right\|^2\le \left\|x-y\right\|^2.
\end{equation}
\item [$\rm ii.$] $z=P_C x$ if and only if 
\begin{equation}\label{eq:EquivalentPC}
\left\langle x-z,z-y \right\rangle \ge 0,\quad \forall y\in C.
\end{equation}
\end{itemize}
\end{lemma}
Any Hilbert space satisfies Opial's conditionc \cite{O1967}, i.e., if $\left\{x_n\right\}\subset H$ converges weakly to $x$ then 
$$ \lim\inf_{n\to\infty}||x_n-x||< \lim\inf_{n\to\infty}||x_n-y||,~\forall y\in H,~y\ne x.$$
\section{Main results} \label{main}
In this section, we present our algorithms and prove their convergence. Without loss of generality, we assume that all bifunctions $f_i:C\times C\to \Re$ 
satisfing Lipschitz-type continuous condition with same constants $c_1,c_2$. Indeed, if $f_i$ is Lipschitz-type continuous with two constants $c_1^i,c_2^i$ 
then we set $c_1=\max\left\{c_1^i:i=1,\ldots,N\right\}$ and $c_2=\max\left\{c_2^i:i=1,\ldots,N\right\}$. From the definition of the Lipschitz-type continuity, 
$f_i$ is also Lipschitz-continuous with the constants $c_1,c_2$. We denote the solution set of SEP for $\left\{f_i\right\}_{i=1}^N$ and 
$\left\{F_j\right\}_{j=1}^M$ by 
$$ 
\Omega=\left\{x^*\in \cap_{i=1}^NEP(f_i,C): Ax^*\in \cap_{j=1}^M EP(F_j,Q) \right\}
 $$
and assume that $\Omega$ is nonempty. We start with the following algorithm.
\begin{algorithm}\label{algor1}(Parallel extragradient-proximal method for SEPs)\\
\textbf{Initialization.} Chose $x_0\in C,~C_0=C$. The control parameters $\lambda,\mu,r_n$ satisfy the following conditions 
$$0<\lambda<\min\left\{\frac{1}{2c_1},\frac{1}{2c_2}\right\},~r_n\ge d>0,~0<\mu<\frac{2}{||A||^2}.$$
\textbf{Step 1.} Solve $N$ strongly convex optimization programs in parallel
$$
\begin{cases}
y_n^i=\arg\min\left\{\lambda f_i(x_n,y)+\frac{1}{2}||y-x_n||^2:y\in C\right\},i=1,\ldots,N,\\
z_n^i=\arg\min\left\{\lambda f_i(y_n^i,y)+\frac{1}{2}||y-x_n||^2:y\in C\right\},i=1,\ldots,N.
\end{cases}
$$
\textbf{Step 2.} Find among $z_n^i$ the furthest element from $x_n$, i.e.,
$$ \bar{z}_n=\arg\max\left\{||z_n^i-x_n||:i=1,\ldots,N\right\}. $$
\textbf{Step 3.} Solve $M$ strongly monotone regularized equilibrium programs in parallel
$$ w_n^j=T_{r_n}^{F_j}(A\bar{z}_n),j=1,\ldots,M. $$
\textbf{Step 4.} Find among $w_n^j$ the furthest element from $A\bar{z}_n$, i.e.,
$$ \bar{w}_n=\arg\max\left\{||w_n^j-A\bar{z}_n||:j=1,\ldots,M\right\}. $$
\textbf{Step 5.} Compute $x_{n+1}=P_C\left(\bar{z}_n+\mu A^*(\bar{w}_n-A\bar{z}_n)\right)$. 
Set $n=n+1$ and go back \textbf{Step 1}.
\end{algorithm}
We need the following lemma to prove the convergence of Algorithm \ref{algor1}.
\begin{lemma}\cite[Lemma 3.1]{A2013} (cf. \cite[Theorem 3.2]{QMH2008})\label{lem1}
Suppose that $x^*\in \cap_{i=1}^NEP(f_i,C)$ and $\left\{x_n\right\}$, $\left\{y_n^i\right\}$, $\left\{z_n^i\right\}$ are 
the sequences generated by Algorithm \ref{algor1}. Then
\begin{itemize}
\item [$\rm i.$] $\lambda \left(f_i(x_n,y)-f_i(x_n,y_n^i)\right) \ge \left\langle y_n^i-x_n, y_n^i-y\right\rangle, \forall y\in C.$
\item [$\rm ii.$]$||z_n^i-x^*||^2\le ||x_n-x^*||^2-(1-2\lambda c_1)||y_n^i-x_n||^2-(1-2\lambda c_2)||y_n^i-z_n^i||^2.$
\end{itemize}
\end{lemma}
\begin{theorem}[Weak convergence theorem]\label{theo1}
Let $C,Q$ be two nonempty closed convex subsets of two real Hilbert spaces $H_1$ and $H_2$, respectively. Let 
$\left\{f_i\right\}_{i=1}^N:C\times C\to \Re$ be a finite family of bifunctions satisfying Condition 1 and 
$\left\{F_j\right\}_{j=1}^M:Q\times Q\to \Re$ be a finite family 
of bifunctions satisfying Condition 2. Let $A:H_1\to H_2$ be a 
bounded linear operator with the adjoint $A^*$. In addition the solution set $\Omega$
is nonempty. Then, the sequences $\left\{x_n\right\}$, $\left\{y^i_n\right\}$, $\left\{z^i_n\right\}$ generated by Algorithm \ref{algor1} 
converge weakly to some point 
$p\in \cap_{i=1}^NEP(f_i,C)$ and $\left\{w^j_n\right\}$ converges weakly to $Ap\in \cap_{j=1}^M EP(F_j,Q)$.
\end{theorem}
\begin{proof} We divide the proof of Theorem \ref{theo1} into three claims.\\
\textbf{Claim 1.} There exists the limit of the sequence $\left\{||x_n-x^*||\right\}$ for all $x^*\in \Omega$.\\
\textit{The proof of Claim 1.} From Lemma \ref{lem1}.ii. and the hypothesis of $\lambda$, we have $||z_n^i-x^*||\le ||x_n-x^*||$ for all $x^*\in \Omega$. Thus,
\begin{equation}\label{eq:2*}
||\bar{z}_n-x^*||\le ||x_n-x^*||.
\end{equation}
Suppose $j_n\in \left\{1,\ldots,M\right\}$ such that $\bar{w}_n=w_n^{j_n}$. From Lemma \ref{lem.CH2005}(B2), we have
\begin{align}
||\bar{w}_n-Ax^*||^2&=||T_{r_n}^{F_{j_n}}(A\bar{z}_n)-T_{r_n}^{F_{j_n}}(Ax^*)||^2\notag\\ 
&\le \left\langle T_{r_n}^{F_{j_n}}(A\bar{z}_n)-T_{r_n}^{F_{j_n}}(Ax^*),A\bar{z}_n-Ax^*\right\rangle \notag\\
&= \left\langle \bar{w}_n-Ax^*,A\bar{z}_n-Ax^*\right\rangle \notag\\
&=\frac{1}{2}\left\{||\bar{w}_n-Ax^*||^2+||A\bar{z}_n-Ax^*||^2-||\bar{w}_n-A\bar{z}_n||^2\right\}.\notag
\end{align}
Thus,
\begin{equation*}
||\bar{w}_n-Ax^*||^2\le ||A\bar{z}_n-Ax^*||^2-||\bar{w}_n-A\bar{z}_n||^2
\end{equation*}
or
\begin{equation*}
||\bar{w}_n-Ax^*||^2- ||A\bar{z}_n-Ax^*||^2\le-||\bar{w}_n-A\bar{z}_n||^2.
\end{equation*}
This together with the following fact
\begin{align*}
\left\langle A(\bar{z}_n-x^*),\bar{w}_n-A\bar{z}_n\right\rangle
=\frac{1}{2}\left\{||\bar{w}_n-Ax^*||^2-||A\bar{z}_n-Ax^*||^2-||\bar{w}_n-A\bar{z}_n||^2\right\}
\end{align*}
implies that 
\begin{equation}\label{eq:3*}
\left\langle A(\bar{z}_n-x^*),\bar{w}_n-A\bar{z}_n\right\rangle\le -||\bar{w}_n-A\bar{z}_n||^2.
\end{equation}
From the definition of $x_{n+1}$ and the nonexpansiveness of the projection,
\begin{align}
||x_{n+1}-x^*||^2&=||P_C\left(\bar{z}_n+\mu A^*(\bar{w}_n-A\bar{z}_n)\right)-P_Cx^*||^2\notag\\ 
&\le||\bar{z}_n-x^*+\mu A^*(\bar{w}_n-A\bar{z}_n)|| ^2\notag\\ 
&=||\bar{z}_n-x^*|| ^2+\mu^2||A^*(\bar{w}_n-A\bar{z}_n)|| ^2+2\mu\left\langle \bar{z}_n-x^*,A^*(\bar{w}_n-A\bar{z}_n)\right\rangle\notag\\ 
&\le||\bar{z}_n-x^*|| ^2+\mu^2||A^*|| ^2||\bar{w}_n-A\bar{z}_n|| ^2+2\mu\left\langle A(\bar{z}_n-x^*),\bar{w}_n-A\bar{z}_n\right\rangle\notag\\ 
&\le||\bar{z}_n-x^*|| ^2+\mu^2||A^*|| ^2||\bar{w}_n-A\bar{z}_n|| ^2-2\mu||\bar{w}_n-A\bar{z}_n||^2\notag\\ 
&\le||\bar{z}_n-x^*|| ^2-\mu(2-\mu ||A^*|| ^2)||\bar{w}_n-A\bar{z}_n|| ^2\label{eq:4*}\\ 
&\le||\bar{z}_n-x^*|| ^2\label{eq:5*}
\end{align}
in which the last inequality is followed from the assumption of $\mu$. From the relations (\ref{eq:2*}) and (\ref{eq:5*}),
$$0\le ||x_{n+1}-x^*||\le ||\bar{z}_n-x^*||\le ||x_n-x^*||, ~\forall x^*\in \Omega. $$
Therefore, the sequence $\left\{||x_{n+1}-x^*||\right\}$ is decreasing and so there exist the limits
\begin{equation}\label{eq:6*}
\lim\limits_{n\to\infty}||x_n-x^*||=\lim\limits_{n\to\infty}||\bar{z}_n-x^*||=p(x^*),~\forall x^*\in \Omega.
\end{equation}
\textbf{Claim 2.} $\lim\limits_{n\to\infty}||z^i_n-x_n||=\lim\limits_{n\to\infty}||y^i_n-x_n||=\lim\limits_{n\to\infty}||w^j_n-A\bar{z}_n||=0$.\\
\textit{The proof of Claim 2.} Suppose that $i_n$ is the index in $\left\{1,\ldots,N\right\}$ such that $\bar{z}_n=z_n^{i_n}$. 
From Lemma \ref{lem1}.ii. with $i=i_n$,
$$||\bar{z}_n-x^*||^2\le ||x_n-x^*||^2-(1-2\lambda c_1)||y_n^{i_n}-x_n||^2-(1-2\lambda c_2)||y_n^{i_n}-\bar{z}_n||^2.  $$
Thus
$$(1-2\lambda c_1)||y_n^{i_n}-x_n||^2+(1-2\lambda c_2)||y_n^{i_n}-\bar{z}_n||^2\le ||x_n-x^*||^2-||\bar{z}_n-x^*||^2. $$
This together with (\ref{eq:6*}) and the hypothesis of $\lambda$ implies that
$$ \lim\limits_{n\to\infty}||y_n^{i_n}-x_n||=\lim\limits_{n\to\infty}||y_n^{i_n}-\bar{z}_n||=0. $$
Thus
\begin{equation}\label{eq:7*}
\lim\limits_{n\to\infty}||\bar{z}_n-x_n||=0
\end{equation}
because of $||\bar{z}_n-x_n||\le||y_n^{i_n}-x_n||+||y_n^{i_n}-\bar{z}_n||$. It follows from the last limit and the definition 
of $\bar{z}_n$ that
\begin{equation}\label{eq:8*}
\lim\limits_{n\to\infty}||z^i_n-x_n||=0,~\forall i=1,\ldots,N.
\end{equation}
From Lemma \ref{lem1}.ii. and the triangle inequality,
\begin{align*}
(1-2\lambda c_1)||y_n^{i}-x_n||^2&\le ||x_n-x^*||^2-||{z}^i_n-x^*||^2\\ 
&=\left(||x_n-x^*||-||{z}^i_n-x^*||\right)\left(||x_n-x^*||+|{z}^i_n-x^*||\right)\\
& \le ||x_n-z_n^i||\left(||x_n-x^*||+|{z}^i_n-x^*||\right)
\end{align*}
which implies that
\begin{equation}\label{eq:9*}
\lim\limits_{n\to\infty}||y_n^{i}-x_n||=0
\end{equation}
because of the relation (\ref{eq:8*}), the hypothesis of $\lambda$ and the boundedness of $\left\{x_n\right\},\left\{z_n^i\right\}$. Moreover, 
from (\ref{eq:4*}), we obtain
\begin{equation}\label{eq:10*}
\mu(2-\mu ||A^*|| ^2)||\bar{w}_n-A\bar{z}_n|| ^2\le ||\bar{z}_n-x^*|| ^2-||x_{n+1}-x^*||^2.
\end{equation}
Passing to the limit in the last inequality as $n\to\infty$ and using the relation (\ref{eq:6*}) and $\mu(2-\mu ||A^*|| ^2)>0$, one has
\begin{equation}\label{eq:11*}
\lim\limits_{n\to\infty}||\bar{w}_n-A\bar{z}_n||=0.
\end{equation}
From the definition of $\bar{w}_n$, we obtain
\begin{equation}\label{eq:12*}
\lim\limits_{n\to\infty}||{w}^j_n-A\bar{z}_n||=0,~\forall j=1,\ldots,M.
\end{equation}
\textbf{Claim 3.} $x_n,y_n^i,z_n^i\rightharpoonup p\in \cap_{i=1}^N EP(f_i,C)$ and $w_n^j\rightharpoonup Ap\in \cap_{j=1}^M EP(F_j,Q)$.\\
\textit{The proof of Claim 3.} Since $\left\{x_n\right\}$ is bounded, there exists a subsequence $\left\{x_m\right\}$ of $\left\{x_n\right\}$ 
which converges weakly to $p$. 
Since $C$ is convex, $C$ is weakly closed, and so $p\in C$. Thus, $y_m^i\rightharpoonup p$, $z_m^i\rightharpoonup p$ and 
$A\bar{z}_m\rightharpoonup Ap$, $w_m^j\rightharpoonup Ap$ 
because of the relations (\ref{eq:8*}), (\ref{eq:11*}) and (\ref{eq:12*}). It follows from Lemma \ref{lem1}.i. that 
$$ \lambda \left(f_i(x_m,y)-f_i(x_m,y_m^i)\right) \ge \left\langle y_m^i-x_m, y_m^i-y\right\rangle, \forall y\in C.$$
Passing to the limit in the last inequality as $m\to\infty$ and using the hypothesis $\rm (A3)$ and $\lambda>0$, we obtain $f_i(p,y)\ge 0,~\forall y\in C$. 
Thus, $p\in \cap_{i=1}^N EP(f_i,C)$. Now, we show that $Ap\in \cap_{j=1}^M EP(F_j,Q)$. By Lemma \ref{lem.CH2005}, $EP(F_j,Q)=Fix(T_r^{F_j})$ 
for some $r>0$. Assume that $Ap\notin Fix(T_r^{F_j})$, i.e., $Ap\ne T_r^{F_j}(Ap)$. By Opial's condition in $H$, the relation (\ref{eq:12*}) and Lemma 
\ref{lem.H2012}, we have
\begin{align*}
\lim\inf_{m\to\infty}||A\bar{z}_m-Ap||&<\lim\inf_{m\to\infty}||A\bar{z}_m-T_r^{F_j}(Ap)||\\ 
&\le \lim\inf_{m\to\infty}\left[||A\bar{z}_m-T_{r_m}^{F_j}(A\bar{z}_m)||+||T_{r_m}^{F_j}(A\bar{z}_m)-T_r^{F_j}(Ap)||\right]\\
&= \lim\inf_{m\to\infty}||T_{r_m}^{F_j}(A\bar{z}_m)-T_r^{F_j}(Ap)||\\
&= \lim\inf_{m\to\infty}||T_r^{F_j}(Ap)-T_{r_m}^{F_j}(A\bar{z}_m)||\\
&\le \lim\inf_{m\to\infty}\left[||Ap-A\bar{z}_m||+\frac{|r-r_m|}{r_m}||T_{r_m}^{F_j}(A\bar{z}_m)-A\bar{z}_m||\right]\\
&= \lim\inf_{m\to\infty}||Ap-A\bar{z}_m||.
\end{align*}
This is contrary. Thus, $Ap\in Fix(T_r^{F_j})=EP(F_j,Q)$, i.e., $Ap\in \cap_{j=1}^M EP(F_j,Q)$.

Finally, we show that the whole sequence $\left\{x_n\right\}$ converges weakly to $p$. Indeed, suppose that $\left\{x_n\right\}$ has a subsequence 
$\left\{x_k\right\}$ which converges weakly to $q\ne p$. By Opial's condition in $H$, we have
\begin{align*}
\lim\inf_{k\to\infty}||x_k-q||&<\lim\inf_{k\to\infty}||x_k-p||=\lim\inf_{m\to\infty}||x_m-p||\\ 
&<\lim\inf_{m\to\infty}||x_m-q||=\lim\inf_{k\to\infty}||x_k-q||. 
\end{align*}
This is a contradiction. Thus, the whole sequence $\left\{x_n\right\}$ converges weakly to $p$. By Claim 2, $y^i_n,z_n^i\rightharpoonup p$  and 
$w_n^j\rightharpoonup Ap$ as $n\to \infty$. Theorem \ref{theo1} is proved.
\end{proof}
\begin{corollary}\label{cor1}
Let $C,Q$ be two nonempty closed convex subsets of two real Hilbert spaces $H_1$ and $H_2$, respectively. Let 
$f:C\times C\to \Re$ be a bifunction satisfying Condition 1 and 
$F:Q\times Q\to \Re$ be a bifunction satisfying Condition 2. Let $A:H_1\to H_2$ be a 
bounded linear operator with the adjoint $A^*$. In addition the solution set 
$\Omega=\left\{x^*\in EP(f,C): Ap\in EP(F,Q) \right\}$
is nonempty. Let $\left\{x_n\right\}$, $\left\{y_n\right\}$, $\left\{z_n\right\}$ and $\left\{w_n\right\}$ 
be the sequences generated by the following manner: $x_0\in C,~C_0=C$ and 
$$
\begin{cases}
y_n=\arg\min\left\{\lambda f(x_n,y)+\frac{1}{2}||y-x_n||^2:y\in C\right\},\\
z_n=\arg\min\left\{\lambda f_i(y_n,y)+\frac{1}{2}||y-x_n||^2:y\in C\right\},\\
w_n=T_{r_n}^{F}(A{z}_n),\\
x_{n+1}=P_C\left({z}_n+\mu A^*({w}_n-A{z}_n)\right),
\end{cases}
$$
where $\lambda, r_n,\mu$ satisfy the conditions in Theorem \ref{theo1}. Then, the 
sequences $\left\{x_n\right\}$, $\left\{y_n\right\}$, $\left\{z_n\right\}$ converge weakly to some point 
$p\in EP(f,C)$ and $\left\{w_n\right\}$ converges weakly to $Ap\in EP(F,Q)$.
\end{corollary}
\begin{proof}
Corollary \ref{cor1} is directly followed from Theorem \ref{theo1} with $f_i=f$ and $F_j=F$ for all $i,j$.
\end{proof}
In order to obtain an algorithm which provides the strong convergence, we propose the following parallel hybrid extragradient-proximal method 
that combines Algorithm \ref{algor1} with the hybrid (outer approximation) method.
\begin{algorithm}\label{algor2}(Parallel hybrid extragradient-proximal method for SEPs)\\
\textbf{Initialization.} Chose $x_0\in C,~C_0=C$, the control parameters $\lambda, r_n,\mu$ satisfy the following conditions
$$0<\lambda<\min\left\{\frac{1}{2c_1},\frac{1}{2c_2}\right\},~r_n\ge d>0, ~0<\mu<\frac{2}{||A||^2}.$$
\textbf{Step 1.} Solve $N$ strongly convex optimization programs in parallel
$$
\begin{cases}
y_n^i=\arg\min\left\{\lambda f_i(x_n,y)+\frac{1}{2}||y-x_n||^2:y\in C\right\},i=1,\ldots,N,\\
z_n^i=\arg\min\left\{\lambda f_i(y_n^i,y)+\frac{1}{2}||y-x_n||^2:y\in C\right\},i=1,\ldots,N.
\end{cases}
$$
\textbf{Step 2.} Find among $z_n^i$ the furthest element from $x_n$, i.e.,
$$ \bar{z}_n=\arg\max\left\{||z_n^i-x_n||:i=1,\ldots,N\right\}. $$
\textbf{Step 3.} Solve $M$ strongly monotone regularized equilibrium programs in parallel
$$ w_n^j=T_{r_n}^{F_j}(A\bar{z}_n),j=1,\ldots,M. $$
\textbf{Step 4.} Find among $w_n^j$ the furthest element from $A\bar{z}_n$, i.e.,
$$ \bar{w}_n=\arg\max\left\{||w_n^j-A\bar{z}_n||:j=1,\ldots,M\right\}. $$
\textbf{Step 5.} Compute $t_n=P_C\left(\bar{z}_n+\mu A^*(\bar{w}_n-A\bar{z}_n)\right)$.\\
\textbf{Step 6.} Compute $x_{n+1}=P_{C_{n+1}}(x_0)$, where $C_{n+1}=\left\{v\in C_n:||t_n-v||\le ||\bar{z}_n-v||\le ||x_n-v||\right\}$. 
Set $n=n+1$ and go back \textbf{Step 1}.
\end{algorithm}
We have the following result.
\begin{theorem}[Strong convergence theorem]\label{theo2}
Let $C,Q$ be two nonempty closed convex subsets of two real Hilbert spaces $H_1$ and $H_2$, respectively. Let 
$\left\{f_i\right\}_{i=1}^N:C\times C\to \Re$ be a finite family of bifunctions satisfying Condition 1 and 
$\left\{F_j\right\}_{j=1}^M:Q\times Q\to \Re$ be a finite family 
of bifunctions satisfying Condition 2. Let $A:H_1\to H_2$ be a 
bounded linear operator with the adjoint $A^*$. In addition the solution set $\Omega$ is nonempty. Then, the 
sequences $\left\{x_n\right\}$, $\left\{y^i_n\right\}$, $\left\{z^i_n\right\}$ generated by Algorithm \ref{algor2} converge 
strongly to $x^\dagger=P_\Omega(x_0)$ and $\left\{w^j_n\right\}$ converges strongly to $Ax^\dagger\in \cap_{j=1}^M EP(F_j,Q)$.
\end{theorem}
\begin{proof} We also divide the proof of Theorem \ref{theo2} into several claims.\\
\textbf{Claim 1.} $C_n$ is closed convex set and $\Omega\subset C_n$ for all $n\ge 0$.\\
\textit{The proof of Claim 1.}
Set 
$$ 
\begin{cases}
C_n^1=\left\{v\in H_1:||t_n-v||\le ||\bar{z}_n-v||\right\},\\
C_n^2=\left\{v\in H_1:||\bar{z}_n-v||\le ||x_n-v||\right\}.
\end{cases}
$$
Then
\begin{equation}\label{eq:1}
C_{n+1}=C_n\cap C_n^1\cap C_n^2.
\end{equation}
Note that $C_n^1,C_n^2$ are either the halfspaces or the whole space $H_1$ for all $n\ge 0$. Hence, they are closed and convex. Obviously, 
$C_0=C$ is closed and convex. Suppose that $C_n$ is closed and convex for some $n\ge 0$. Then, from (\ref{eq:1}), $C_{n+1}$ is also 
closed and convex. By the induction, $C_{n}$ is closed and convex for all $n\ge 0$. Next, we show that $\Omega\subset C_n$ for all $n\ge 0$. 
From Lemma \ref{lem1}.ii. and the hypothesis of $\lambda$, we have $||z_n^i-x^*||\le ||x_n-x^*||$ for all $x^*\in \Omega$. Thus,
\begin{equation}\label{eq:2}
||\bar{z}_n-x^*||\le ||x_n-x^*||.
\end{equation}
By arguing similarly to Claim 1 in the proof of Theorem \ref{theo1} we obtain
\begin{align}
||t_n-x^*||^2
&\le||\bar{z}_n-x^*|| ^2-\mu(2-\mu ||A^*|| ^2)||\bar{w}_n-A\bar{z}_n|| ^2\label{eq:4}\\ 
&\le||\bar{z}_n-x^*|| ^2\label{eq:5}. 
\end{align}
From (\ref{eq:2}) and (\ref{eq:5}),
$$ ||t_n-x^*||\le ||\bar{z}_n-x^*||\le ||x_n-x^*||, ~\forall x^*\in \Omega. $$
Thus, by the definition of $C_n$ and the induction, $\Omega\subset C_n$ for all $n\ge 0$.\\
\textbf{Claim 2.} $\left\{x_n\right\}$ is a Cauchy sequence and
$$\lim\limits_{n\to\infty}x_n=\lim\limits_{n\to\infty}y^i_n=\lim\limits_{n\to\infty}z^i_n=p,~\lim\limits_{n\to\infty}w_n^j=\lim\limits_{n\to\infty}A\bar{z}_n=Ap.$$
\textit{The proof of Claim 2.} 
From $x_n=P_{C_n}(x_0)$ and Lemma \ref{lem.PC}.i.,
\begin{equation}\label{eq:6}
||x_n-x_0||\le||u-x_0||,~\forall u\in C_n.
\end{equation}
Therefore, $||x_n-x_0||\le||x_{n+1}-x_0||$ because $x_{n+1}\in C_{n+1}\subset C_n$. This implies that the sequence $\left\{||x_n-x_0||\right\}$ 
is non-decreasing. The inequality (\ref{eq:6}) with $u=x^\dagger:=P_\Omega(x_0)\in \Omega\subset C_n$ leads to 
\begin{equation}\label{eq:7}
||x_n-x_0||\le||x^\dagger-x_0||.
\end{equation}
Thus, the sequence $\left\{||x_n-x_0||\right\}$ is bounded, and so there exists the limit of $\left\{||x_n-x_0||\right\}$. For all $m\ge n$, from the 
definition of $C_m$, we have $x_m\in C_m\subset C_n$. So, from $x_n=P_{C_n}(x_0)$ and Lemma \ref{lem.PC}.i.,
\begin{equation}\label{eq:8}
||x_n-x_m||^2\le||x_m-x_0||^2-||x_n-x_0||^2.
\end{equation}
Passing to the limit in the last inequality as $m,n\to\infty$, we get
\begin{equation}\label{eq:10}
\lim\limits_{m,n\to\infty}||x_n-x_m||=0.
\end{equation}
Thus, $\left\{x_n\right\}$ is a Cauchy sequence and 
\begin{equation}\label{eq:11}
\lim\limits_{n\to\infty}||x_n-x_{n+1}||=0.
\end{equation}
From the definition of $C_{n+1}$ and $x_{n+1}\in C_{n+1}$, we have
\begin{equation*}
||t_n-x_{n+1}||\le ||\bar{z}_n-x_{n+1}||\le ||x_n-x_{n+1}||.
\end{equation*}
Thus, from the triangle inequality, one has
\begin{align*}
&||t_n-x_{n}||\le||t_n-x_{n+1}||+||x_{n+1}-x_n||\le2||x_n-x_{n+1}||,\\ 
&||\bar{z}_n-x_{n}||\le||\bar{z}_n-x_{n+1}||+||x_{n+1}-x_n||\le2||x_n-x_{n+1}||,\\
&||\bar{z}_n-t_{n}||\le||\bar{z}_n-x_{n}||+||x_{n}-t_n||\le4||x_n-x_{n+1}||
\end{align*}
Three last inequalities together with the relation (\ref{eq:11}) imply that
\begin{equation}\label{eq:12}
\lim\limits_{n\to\infty}||t_n-x_n||=\lim\limits_{n\to\infty}||\bar{z}_n-t_{n}||=\lim\limits_{n\to\infty}||\bar{z}_n-x_n||=0.
\end{equation}
Hence, from the definition of $\bar{z}_n$, we also obtain
\begin{equation}\label{eq:13}
\lim\limits_{n\to\infty}||z_n^i-x_n||=0,~\forall i=1,\ldots,N.
\end{equation}
Since $\left\{x_n\right\}$ is a Cauchy sequence, $x_n\to p$ and
\begin{equation}\label{eq:14}
\lim\limits_{n\to\infty}t_n=\lim\limits_{n\to\infty}\bar{z}_n=\lim\limits_{n\to\infty}z_n^i=p,~\forall i=1,\ldots,N,
\end{equation}
and so
\begin{equation}\label{eq:14}
\lim\limits_{n\to\infty}A\bar{z}_n=Ap.
\end{equation}
From the relation (\ref{eq:4}) and the triangle inequality, we obtain
\begin{align*}
\mu(2-\mu ||A^*|| ^2)||\bar{w}_n-A\bar{z}_n|| ^2&\le||\bar{z}_n-x^*||^2-||t_n-x^*||^2\\ 
& =(||\bar{z}_n-x^*||-||t_n-x^*||)(||\bar{z}_n-x^*||+||t_n-x^*||)\\
&\le ||\bar{z}_n-t_n||(||\bar{z}_n-x^*||+||t_n-x^*||)
\end{align*}
Thus, from $\mu(2-\mu ||A^*|| ^2)>0$, the boundedness of $\left\{t_n\right\},\left\{\bar{z}_n\right\}$ and (\ref{eq:12}) we obtain
\begin{equation*}
\lim\limits_{n\to\infty}||\bar{w}_n-A\bar{z}_n||=0.
\end{equation*}
From the definition of $\bar{w}_n$, we get
\begin{equation}\label{eq:15}
\lim\limits_{n\to\infty}||w^j_n-A\bar{z}_n||=0,~\forall j=1,\ldots,M,
\end{equation}
which follows from (\ref{eq:14}) that
\begin{equation}\label{eq:16}
\lim\limits_{n\to\infty}w^j_n=Ap,~\forall j=1,\ldots,M.
\end{equation}
From Lemma \ref{lem1}.ii. and the triangle inequality, we have
\begin{align*}
(1-2\lambda c_1)||y_n^i-x_n||^2&\le||x_n-x^*||^2-||z_n^i-x^*||^2\\ 
& =(||x_n-x^*||-||z^i_n-x^*||)(||x_n-x^*||+||z^i_n-x^*||)\\
&\le ||x_n-z^i_n||(||x_n-x^*||+||z^i_n-x^*||)
\end{align*}
Thus, from the hypothesis of $\lambda$, the boundedness of $\left\{x_n\right\},\left\{z^i_n\right\}$ and (\ref{eq:13}) we obtain
\begin{equation}\label{eq:16*}
\lim\limits_{n\to\infty}||y_n^i-x_n||=0.
\end{equation}
Therefore, $y_n^i\to p$ as $n\to\infty$.\\
\textbf{Claim 3.} $p\in \Omega$ and $p=x^\dagger:=P_\Omega(x_0)$.\\
\textit{The proof of Claim 3.} 
By Lemma \ref{lem1}.ii, we get
$$\lambda \left(f_i(x_n,y)-f_i(x_n,y_n^i)\right) \ge \left\langle y_n^i-x_n, y_n^i-y\right\rangle, \forall y\in C.$$
Passing to the limit in the last inequality as $n\to \infty$ and using the assumption $\rm (A3)$ and the relation (\ref{eq:13}), we get 
$f_i(p,y)\ge 0$ for all $y\in C$. Thus, $p\in \cap_{i=1}^N EP(f_i,C)$. 

Moreover, from Lemma \ref{lem.H2012}, for some $r>0$ we have
\begin{align}
||T_r^{F_j}(Ap)-Ap||&\le||T_r^{F_j}(Ap)-T_{r_n}^{F_j}(A\bar{z}_n)||+||T_{r_n}^{F_j}(A\bar{z}_n)-A\bar{z}_n||+||A\bar{z}_n-Ap||\notag\\ 
& \le||Ap-A\bar{z}_n||+\frac{r_n-r}{r_n}||T_{r_n}^{F_j}(A\bar{z}_n)-A\bar{z}_n||+||T_{r_n}^{F_j}(A\bar{z}_n)-A\bar{z}_n||+||A\bar{z}_n-Ap||\notag\\ 
&=2||Ap-A\bar{z}_n||+\frac{r_n-r}{r_n}||w_n^j-A\bar{z}_n||+||w_n^j-A\bar{z}_n||\to 0\notag
\end{align}
which is followed from the relations (\ref{eq:14}),(\ref{eq:15}),(\ref{eq:16}) and $r_n\ge d>0$. Thus, $T_r^{F_j}(Ap)-Ap=0$ or $Ap$ is 
a fixed point of $T_r^{F_j}$. From Lemma \ref{lem.CH2005}, we obtain $Ap\in \cap_{j=1}^M EP({F_j},Q)$. Thus, $p\in \Omega$. Finally, from (\ref{eq:7}), 
$||x_n-x_0||\le||x^\dagger-x_0||$ where $x^\dagger=P_\Omega(x_0)$. Taking $n\to \infty$ in this inequality, one has $||p-x_0||\le||x^\dagger-x_0||$. 
From the definition of $x^\dagger$, $p=x^\dagger$. Theorem \ref{theo2} is proved.
\end{proof}
\begin{corollary}\label{cor2}
Let $C,Q$ be two nonempty closed convex subsets of two real Hilbert spaces $H_1$ and $H_2$, respectively. Let 
$f:C\times C\to \Re$ be a bifunction satisfying Condition 1 and 
$F:Q\times Q\to \Re$ be a bifunction satisfying Condition 2. Let $A:H_1\to H_2$ be a 
bounded linear operator with the adjoint $A^*$. In addition the solution set 
$\Omega=\left\{x^*\in EP(f,C): Ap\in EP(F,Q) \right\}$
is nonempty. Let $\left\{x_n\right\}$, $\left\{y_n\right\}$, $\left\{z_n\right\}$, $\left\{t_n\right\}$ and $\left\{w_n\right\}$ 
be the sequences generated by the following manner: $x_0\in C,~C_0=C$ and 
$$
\begin{cases}
y_n=\arg\min\left\{\lambda f(x_n,y)+\frac{1}{2}||y-x_n||^2:y\in C\right\},\\
z_n=\arg\min\left\{\lambda f(y_n,y)+\frac{1}{2}||y-x_n||^2:y\in C\right\},\\
w_n=T_{r_n}^{F}(A{z}_n),\\
t_{n}=P_C\left({z}_n+\mu A^*({w}_n-A{z}_n)\right),\\
C_{n+1}=\left\{v\in C_n:||t_n-v||\le ||\bar{z}_n-v||\le ||x_n-v||\right\},\\ 
x_{n+1}=P_{C_{n+1}}(x_0),
\end{cases}
$$
where $\lambda, r_n,\mu$ satisfy the conditions in Theorem \ref{theo2}. Then, the 
sequences $\left\{x_n\right\}$, $\left\{y_n\right\}$, $\left\{z_n\right\}$, $\left\{t_n\right\}$ converge strongly to 
$x^\dagger=P_\Omega(x_0)$ and $\left\{w_n\right\}$ converges strongly to $Ap\in EP(F,Q)$.
\end{corollary}
\begin{proof}
Corollary \ref{cor2} is directly followed from Theorem \ref{theo2} with $f_i=f$ and $F_j=F$ for all $i,j$.
\end{proof}
\section*{Conclusions}
We have proposed two parallel extragradient-proximal algorithms for split equilibrium problems and proved their convergence. 
we have designed the algorithms by combining the extragradient method for a class of pseudomonotone and Lipschitz-type continuous 
bifunctions, the proximal method for monotone bifunctions and the hybrid (outer approximation) method. 

\end{document}